\documentclass[11pt]{article}
\usepackage{graphicx}
\usepackage{amssymb}
\usepackage{epstopdf}
\usepackage[all]{xy}
 \newenvironment{proof}  {\par\noindent{\bf Proof}\ }
                                  {\hfill$\Box$\par\medskip}

\newcommand{\Ker}{\mathop{\mathrm{Ker}}\nolimits}
\newcommand{\image}{\mathop{\mathrm{Im}}\nolimits}
\title{Binary operations for homotopy groups with coefficients
 }
 
\author{ARKOWITZ Martin}

\date{ }
\begin{document}

\maketitle
{\it Department of Mathematics, Dartmouth College, Hanover, NH 03755, USA;
\newline Email: martin.arkowitz@dartmouth.edu}

Accepted for publication in Science China - Mathematics 

{\bf Abstract}   We define and study binary operations for homotopy groups with coefficients. We give conditions to prove that certain binary operations are the 
homomorphic image of the generalized Whitehead product. This allows carrying over properties of the generalized Whitehead product to these 
operations. We discuss two classes of binary operations, the Whitehead products and the Torsion
products. We introduce a new class of operations called Ext operations 
and determine some of its properties. We compare
the Torsion product to the Whitehead product in a special case. We prove that the smash product of two Moore spaces has the homotopy type of a wedge of two Moore spaces.
\smallskip


\bigskip

\newtheorem{prop}{Proposition}[section]
\newtheorem{lemma}[prop]{Lemma}
\newtheorem{theorem}[prop]{Theorem}
\newtheorem{corollary}[prop]{Corollary}
\newtheorem{rmk}[prop]{Remark}
\newtheorem{definition}[prop]{Definition}
 
\section{Introduction}\noindent  Products, such as the cup product for cohomology groups or the Whitehead product for homotopy groups, are important binary operations in 
Algebraic Topology. The cup product is defined for cohomology groups with coefficients, whereas the Whitehead product is usually defined for integral homotopy groups. In spite of 
the fact that the concept of homotopy groups with coefficients has been available for some time, there has been very little work on binary operations for homotopy groups with 
coefficients (two exceptions are \cite{h} and  \cite{n}). Our object in this paper is to discuss in some detail such binary operations and to provide a unifying method for studying them. 
\medskip

\noindent The following is a brief outline of the paper. We shall refer to homotopy groups with coefficients as homotopy groups and binary operations as operations.  After a preliminary section, 
the generalized Whitehead product and generalized Whitehead product map are recalled. We show that the map can be extended to a map of a cone into a product of suspensions. We 
next define the general notion of an operation for homotopy groups. A class of operations, called basic operations, is defined and it is proved that these are homomorphic images of the 
generalized Whitehead product. We consider operations that have been restricted further (special operations) with a view to
studying two particular classes of operations, the Whitehead products and the Torsion products. In the case when an operation is the image of a generalized Whitehead product, we give 
necessary and sufficient conditions for it to be a special operation. This is then applied to Whitehead and Torsion
products. In the final section, we discuss a number of topics related to earlier sections. We first 
compare the Torsion product to Neisendorfer's Whitehead product \cite{n} when all coefficient groups are cyclic of the same odd prime power order. We consider in more detail the class of operations called Ext operations which were introduced earlier in the paper.
Next we establish a homotopy equivalence between the smash product of two Moore spaces and the wedge of two (different) Moore spaces. Finally we briefly discuss the difference between using Moore spaces or Co-Moore spaces to obtain coefficients.

\section{Preliminaries} \noindent In this section we present our notation and assumptions. All spaces are assumed to be based and of the homotopy type of based CW-complexes and all groups are assumed to be abelian. Maps and homotopies are to preserve base points. The base point is generically denoted by $*$. We let $[f]$ denote the homotopy class of the map $f$ and $f\simeq g$ signifies that $f$ and $g$ are homotopic,  but notationally we often ignore the distinction between maps and homotopy classes. For example, an expression containing a mix of maps and homotopy classes refers to the homotopy class determined by the expression. We write $\Sigma X$ for the (reduced) suspension of the space $X$ and $CX  =X\times I/\{ * \}\times I\cup X\times \{ 1 \}$ for the (reduced) cone. Also $X\vee Y$ denotes the wedge and $X \wedge Y$ the smash product. Furthermore, the join $X*Y$ is the quotient of
$X\times Y\times I$ with the equivalence relations $(x,y,0)\sim (x,y',0)$ and $(x,y,1) \sim (x',y,1)$ and base point
given by $\{ * \}\times \{ * \} \times I$. We use $``\approx "$ for isomorphism of groups and $``\equiv "$ for same homotopy type. We let $[X,Y]$ be the set of homotopy classes of maps from $X$ to $Y.$ A map $f$ induces a homomorphism $f_*$ of homology groups and a homomorphism $f_{\#}$ of homotopy groups. 
For homomorphisms of groups $h:G'\to G $ and $k: H\to H'$, we let $h^*: \mathrm{Hom}(G,H)\to \mathrm{Hom}(G',H)$
and $k_*: \mathrm{Hom}(G,H)\to \mathrm{Hom}(G,H')$ be the induced homomorphisms.
We denote by $\mu ' :A*B \to \Sigma (A\wedge B)$ the homotopy equivalence  obtained by collapsing the subset  
$(A\times \{ * \}\times I)\cup (\{ * \}\times B\times I) $ to a point. The homotopy inverse of $\mu '$ is denoted $\mu :\Sigma (A \wedge B) \to A*B.$
Let $G$ be a group and $n$ an integer $\ge 2$. A {\bf Moore space} $M(G,n)$ is a simply-connected space with a single
non-vanishing reduced homology group $G$ in degree $n$. The {\bf $n$th homotopy group of $X$ with coefficients $G$}, denoted
$\pi _n(X;G)$, is defined to be $[M(G,n), X]$. If $(X;A)$ is a pair of spaces then the homotopy group 
$\pi _n(X,A;G)$ is defined as the set of homotopy classes of maps $(CM(G,n-1),M(G,n-1))\to (X,A).$
We shall refer several times to the Universal Coefficient Theorem for homotopy groups:
\smallskip

\noindent{\bf Theorem} There is a short exact sequence
$$\xymatrix{0\ar[r] &\mathrm{Ext}(G,\pi _{n+1}(X)) \ar[r]^-{\lambda } & \pi _n(X;G) \ar[r]^-{\eta  } & \mathrm{Hom}(G,\pi _n(X))\ar[r] & 0}$$
where $\eta [f] = f_{\#} :G\to \pi _n(X)$ (\cite{h}, p.\,30). When $X$ is replaced by a pair of spaces, the sequence is also exact.

 \section{Generalized Whitehead products}

\noindent Let $A$, $B$, and $X$ be spaces and let $\alpha \in [\Sigma A,X]$ and $\beta \in [\Sigma B,X]$. The generalized Whitehead product \cite{a} of 
$\alpha $ and $\beta $ is an element $[\alpha ,\beta]\in [\Sigma (A\wedge B), X]$  and is defined as follows. 
Let $\alpha $ be represented by $f:\Sigma A \to X$ and $\beta $ represented by $g:\Sigma B \to X$
and let $p_1:A\times B \to A$ and $p_2:A\times B \to B$ be the projections. Then $f' = f(\Sigma p_1)$ and $g' = g(\Sigma p_2)$ map 
$\Sigma (A\times B)$ to $X.$ We define $c =(f',g') = f'^{-1}g'^{-1}f'g',$ the commutator of $f'$ and $g'.$
Let $j:A \vee B \to A\times B$ be the inclusion map and $q:A\times B \to A\wedge B$ the quotient map. 
Clearly $(\Sigma j)^*(c) =0.$ Thus there is a unique element $[\alpha ,\beta]$ such that $(\Sigma q)^*[\alpha ,\beta]
=c$.
When $A$ and $B$ are spheres, this is
just the ordinary Whitehead product. Next let $\iota _1 \in [\Sigma A, \Sigma A\vee \Sigma B]$ 
and $\iota _2 \in [\Sigma B, \Sigma A\vee \Sigma B]$ be the inclusions. Then $[\iota _1, \iota _2] \in [\Sigma (A\wedge B), \Sigma A\vee \Sigma B]$
is called the {\bf universal element} for the generalized Whitehead product. If $f$ represents $\alpha $ and $g$ represents $\beta $, then $[\alpha , \beta] 
=(f,g)_{*}[\iota _1, \iota _2]$, where $(f,g) : \Sigma A\vee \Sigma B \to X$ is the map determined by $f$ and $g$ and $(f,g)_{*}$ is the induced map $[\Sigma (A \wedge B), \Sigma A \vee \Sigma B] \to [\Sigma (A \wedge B), X].$ Thus any generalized Whitehead product is the image of
the universal element. We choose a map $\widetilde{k}: \Sigma (A\wedge B) \to \Sigma A\vee \Sigma B$ in the homotopy class $[\iota _1 , \iota _2 ]$ and 
call it the {\bf generalized Whitehead product map}.
\begin{theorem}\label{3.1}  There is a map $\Lambda : C(A \ast B) \to \Sigma A \times \Sigma B$ such that, if $\Lambda | A*B : A\ast B \to \Sigma A\vee \Sigma B$ is denoted by $\lambda $, then 
 \begin{enumerate}
\item $\lambda \mu  \simeq \widetilde{k}:\Sigma (A \wedge B) \to \Sigma A \vee \Sigma B,$
\item $\Lambda $ induces $\overline{\lambda}:\Sigma (A*B) \to \Sigma A \wedge \Sigma B$ such that $\overline{\lambda}\simeq  \sigma (\Sigma \mu '):\Sigma (A* B) \to \Sigma A \wedge \Sigma B$, where $\sigma: \Sigma ^2(A \wedge B) \to \Sigma A \wedge \Sigma B$
is the homeomorphism given by $\sigma ((a,b),t,u) =((a,t),(b,u))$, for $a\in A$, $b\in B$, and $t, u \in I$.
\end{enumerate}
\end{theorem}
\begin{proof} The function $\Lambda $ was defined by D. Cohen (\cite{c}, Theorem 2.4) 
as follows:
$$ \Lambda ((a,b,t),u) =\left\{ 
\begin{array}{ll}
((a,u),(b,1-2t(1-u)))  & {\mathrm{if}} \, \,  0\le t \le \frac{1}{2}\\
((a, 1-2(1-t)(1-u)),(b,u)) & {\mathrm{if}}  \, \, \frac{1}{2} \le t \le 1,
\\
\end{array} \right.
$$
for $a\in A$, $b\in B$, and $t, u \in I$. The proof of (1) is an immediate consequence of Lemma 4.1 and the proof of 
Theorem 2.4 of \cite{a}. For the proof of (2) we define a (linear) homotopy between $\overline{\lambda}$ and $\sigma  (\Sigma \mu ')$:
$$ \Phi _s (x) =\left\{ 
\begin{array}{ll}
((a,(1-s)u +st),(b,(1-s)(1-2t(1-u))+su))  \quad \,  \,\,\,{\mathrm{if}} \,\, 0\le t \le \frac{1}{2}\\
((a,(1-s) (1-2(1-t)(1-u))+st),(b,u)),   \qquad  \quad \quad {\mathrm{if}} \,\, \frac{1}{2} \le t \le 1.
\\
\end{array} \right.
$$
where $x=((a,b,t),u) \in \Sigma(A*B)$ and $s\in I$.

\end{proof}

\section{Binary operations}
\noindent  Let $G_1$, $G_2$, and $G_3$ be groups and $q_1$, $q_2$, and $q_3$ be integers.
A {\bf (binary) operation} of type $\{G_1, G_2, G_3; q_1, q_2, q_3\}$ is function $T$ which, for every space $X$ and integers $q_1, q_2\ge 2$, assigns
to each $\alpha \in \pi _{q_1}(X;G_1)$ and $\beta \in \pi _{q_{2}}(X;G_2)$, an element $T(\alpha , \beta ) \in \pi _{q_3}(X;G_3)$ 
(also written $T_X(\alpha , \beta ) $) such that if $f:X\to Y$ is a map, then $f_{\#}T_X(\alpha , \beta )
= T_Y(f_{\#}(\alpha ), f_{\#}(\beta ))$. In the examples, $q_3$ will be a simple function of $q_1$ and
$q_2$ such as $ q_1 +q_2 +C$, for some constant $C$.
Let $M_i = M(G_i,q_i)$, $i=1,2, $ or $3$, let $\iota _1 \in \pi _{q_1}(M_1\vee M_2;G_1)$ and $\iota _2 \in \pi _{q_2}(M_1\vee M_2;G_2)$
be the inclusions, and let $[f] =\alpha \in \pi _{q_1}(X;G_1)$ and $[g] = \beta \in \pi _{q_{2}}(X;G_2).$ Then $(f,g)_{\#}T(\iota _1, \iota_2)=T(\alpha , \beta )$. We call $T(\iota _1, \iota_2)
 \in \pi _{q_3}(M_1\vee M_2;G_3)$ the {\bf universal element} for $T$.
\smallskip

\noindent Next 
let the set of all operations of type $\{G_1, G_2, G_3; q_1, q_2, q_3\}$ be denoted 
{\bf O} = {\bf O}$\{G_1, G_2, G_3; q_1, q_2, q_3\}$. If $T$ and $T'$ are two such operations, then $T + T'$ defined by $(T+T')(\alpha , \beta ) = T(\alpha , \beta )+T'(\alpha , \beta )$ is
also in {\bf O}. Thus {\bf O} is an abelian group. Furthermore, the function from {\bf O} to $\pi _{q_3}(M_1\vee M_2;G_3)$ which 
sends $T$ to $T(\iota _1, \iota_2)$ is easily seen to be an isomorphism.

\noindent \begin{definition} Let $\partial :\pi _{q_{{3}+1}}(M_1\times M_2, M_1\vee M_2;G_3) \to \pi _{q_3}(M_1\vee M_2;G_3)$
be the boundary homomorphism in the homotopy sequence of $(M_1\times M_2, M_1\vee M_2)$, let $T$ be an operation as above, and assume
that $q_1,q_2 \ge 3$.
Then $T$ is called a {\bf basic operation} if $T(\iota _1, \iota_2) \in \image \,\partial $, the image of 
$\partial $.
\end{definition}
\smallskip

\noindent Next let $ \bar M_1 = M(G_1, q_1 -1)$ and $\bar M_2 = M(G_2, q_2 -1)$, so $\Sigma \bar M_1 = M_1$ and $\Sigma \bar M_2 = M_2$.
\begin{theorem}\label{theta}  If $T$ is a basic operation and $q_3 <q_1+q_2 +\min(q_1,q_2)-3$, then there exists a unique element $\theta _T\in \pi _{q_3}(\Sigma (\bar M_1 \wedge \bar M_2);G_3)$
such that $T(\alpha ,\beta) = [\alpha , \beta] \, \theta _T = h^* [\alpha , \beta]$, 
for every space $X$, where $\alpha \in [\Sigma \bar M_1, X]$, $\beta \in [\Sigma \bar M_2, X]$, $[h] =\theta _T,$ and $h^*:[\Sigma (\bar M_1 \wedge \bar M_2),X]\to \pi _{q_3}(X;G_3)$ is induced by $h$. Furthermore, $h$ is a suspension and so $h^* $ is a homomorphism. 
\end{theorem}
\begin{proof}
Consider the diagram
$$
\xymatrix{\Sigma (\bar M_1\wedge \bar M_2)\ar[r] \ar[d]^{\lambda \mu } & C\Sigma (\bar M_1\wedge \bar M_2) \ar[r]^-{p'}\ar[d]^{\Lambda ( C\mu)} & \Sigma ^2(\bar M_1\wedge \bar M_2) \ar[d]^{\overline{\lambda}  (\Sigma \mu )}\\
\Sigma \bar M_1 \vee \Sigma \bar M_2\ar[r] & \Sigma \bar M_1 \times \Sigma \bar M_2 \ar[r]^p & \Sigma \bar M_1 \wedge \Sigma \bar M_2,}
$$
where each row is a cofiber sequence, the squares commute, and $p$ and $p'$ are projections. By the Blakers-Massey 
Theorem (\cite{h}, p.\,49), if $r<q_1+q_2+\min(q_1,q_2)-3 $, then both 
$p'_{\#}: \pi _{r+1}(C\Sigma (\bar M_1\wedge \bar M_2),\Sigma (\bar M_1\wedge \bar M_2);G_3)\to \pi _{r+1}(\Sigma ^2(\bar M_1\wedge \bar M_2);G_3)$ and 
$p_{\#}: \pi _{r+1}(\Sigma \bar M_1 \times \Sigma \bar M_2,\Sigma \bar M_1\vee
\Sigma \bar M_2;G_3)\to \pi _{r+1}(\Sigma \bar M_1\wedge \Sigma \bar M_2;G_3)$ 
 are isomorphisms. Therefore the exact homotopy sequences with coefficients of the pairs $(C\Sigma (\bar M_1\wedge \bar M_2),
\Sigma (\bar M_1\wedge \bar M_2))$ and $(\Sigma \bar M_1 \times \Sigma \bar M_2,\Sigma \bar M_1 \vee \Sigma \bar M_2)$, together with the homomorphism of the
first sequence into the second sequence determined by the map $\Lambda (C\mu )$, yield the following commutative square
$$
\xymatrix{\pi _{q_3+1}(\Sigma ^2(\bar M_1\wedge \bar M_2);G_3)\ar[r]^-{\delta '} \ar[d]^{(\overline{\lambda } \Sigma \mu )_{\#}=\sigma _{\#} }& \pi _{q_3}(\Sigma (\bar M_1\wedge \bar M_2);G_3) \ar[d]^{(\lambda \mu)_{\#} =\widetilde{k}_{\#}}
\\ \pi _{q_3+1}(\Sigma \bar M_1 \wedge \Sigma \bar M_2;G_3)\ar[r]^-{\delta } & \pi _{q_3}(\Sigma \bar M_1 \vee \Sigma \bar M_2;G_3),}
$$
where
$p_{\#}^{-1}: \pi _{q_3+1}(\Sigma \bar M_1\wedge \Sigma \bar M_2;G_3) \to \pi _{q_3+1}(\Sigma \bar M_1 \times \Sigma \bar M_2,\Sigma \bar M_1\vee
\Sigma \bar M_2;G_3)$ and $\partial : \pi _{q_3+1}(\Sigma \bar M_1 \times \Sigma \bar M_2,\Sigma \bar M_1\vee \Sigma \bar M_2;G_3) \to \pi _{q_3}(\Sigma \bar M_1 \vee \Sigma \bar M_2;G_3).$
Then  $\delta = \partial \, p_{\#}^{-1}$ and $\delta '$ is similarly defined. Clearly $\delta'$ and $\sigma_{\#}$ are isomorphisms. In addition, it follows from the exact sequence
of $(\Sigma \bar M_1 \times \Sigma \bar M_2,\Sigma \bar M_1 \vee \Sigma \bar M_2)$ that $\delta $ is one-one. Thus $\bar{k}_{\#}$ is one-one and
$\image \widetilde{k}_{\#} = \image \delta = \image \partial p_{\#}^{-1} = \image \partial $. But $T(\iota_1,\iota_2) \in \image \partial $.
Thus there is a unique $\theta _T=[h] \in \pi _{q_3}(\Sigma (\bar M_1 \wedge \bar M_2);G_3)$ with $T(\iota_1 ,\iota_2) = [\iota _1, \iota_2] \, \theta _T = h^* [\iota _1, \iota_2]$. If $f:\Sigma \bar M_1 \to X$ and  $g:\Sigma \bar M_2 \to X$ represent $\alpha $ and $\beta $ respectively, then
$[\alpha , \beta] \, \theta _T = (f,g)_{\#}[\iota _1, \iota_2] \, \theta _T= (f,g)_{\#} T(\iota_1,\iota_2) = T(\alpha ,\beta) $. The second assertion of the theorem is a consequence of the generalized suspension theorem since the dimension of $M(G_3, q_3-1)$ is $\le q_3$ and $q_3< q_1+q_2 +\min(q_1,q_2)-3$.
\end{proof}

\noindent \begin{rmk}\label{BM}  {\em If $G_3$ is a free-abelian group, then the conclusion of Theorem \ref{theta}  holds when $q_3 =q_1+q_2 +\min(q_1,q_2)-3$. This is also true for the conclusion of subsequent results in which this strict inequality appears. This is because the Blakers-Massey Theorem holds in this case.}
\end{rmk}
\begin{corollary}\label{equivs} Let $T$ be an operation
of type $\{ G_1, G_2, G_3 ; q_1, q_2, q_3  \}$ and let $\alpha ,\alpha '\in \pi _{q_1}(X;G_1)$ and 
$\beta ,\beta '\in \pi _{q_2}(X;G_2)$. Consider the following statements:
\begin{enumerate}
\item $T$ is basic;
\item $ j_{\#} T(\iota _1,\iota _2)= 0$, where $j: M_1\vee M_2 \to M_1 \times M_2 $ is the inclusion;
\item $T$ is bi-additive: $T(\alpha +\alpha ', \beta ) = T(\alpha , \beta ) +T(\alpha ', \beta )$ and $T(\alpha , \beta +\beta ')=
T(\alpha , \beta )+T(\alpha , \beta ')$;
\item $T(\alpha , 0)=0$ and $T(0,\beta) =0.$

\end{enumerate}
Then (1) $\Longleftrightarrow $ (2) and (3)$\Longrightarrow $ (4) $\Longrightarrow $(2). If in addition
$q_3 < q_1+q_2 +\min (q_1,q_1) -3$, then (2) $\Longrightarrow $ (3), and in this case all four statements are equivalent.
\end{corollary}
\begin{proof}

(1) $\Longleftrightarrow $ (2): This is an immediate consequence of the exactness of the homotopy sequence of the pair $(M_1\times M_2, M_1\vee M_2)$. 
 
(3)$\Longrightarrow $ (4): $ T(\alpha , 0)=T(\alpha , 0+0) =T(\alpha , 0) +T(\alpha , 0)$ and so
$T(\alpha , 0) =0$. $T(0, \beta ) =0$ is similar. 

(4)$\Longrightarrow $ (2): Let $j_k:M_k\to M_1 \times M_2$ be the inclusions and $p_k: M_1 \times M_2
\to M_k$ be the projections, $k=1,2$. Then $$ j_{\#} T(\iota _1,\iota _2) = T(j_1, j_2) \in \pi _{q_3}(M_1 \times M_2; G_3).$$ But $T(j_1, j_2) =0  \Longleftrightarrow   p_{1\# }T(j_1, j_2) =0 \mathrm{\; and \;  } 
p_{2\# }T(j_1, j_2) =0. $ However by (4), $p_{1\# }T(j_1, j_2) =T(p_1j_1, 0)=0$ and similarly $p_{2\# }T(j_1, j_2) =0.$ This proves (2).

(2) $\Longrightarrow $ (3): Since $q_3 < q_1+q_2 +\min (q_1,q_1) -3$ and since $T$ is basic, (3) follows from Theorem \ref{theta} and the bi-additivity of the generalized Whitehead product (\cite{a}, p.\,14).
\end{proof}

\noindent Next let {\bf BO} = {\bf BO}$\{G_1, G_2, G_3; q_1, q_2, q_3\}$ be the set of all basic operations of type $\{G_1, G_2, G_3; q_1, q_2, q_3\}$. Clearly {\bf BO} $\subseteq $ {\bf O} is a subgroup.
\begin{corollary} If $q_3 <q_1+q_2 +\min(q_1,q_2)-3$, then there is an isomorphism from {\bf BO} to $\pi _{q_3}(\Sigma (\bar M_1 \wedge \bar M_2);G_3)$.
\end{corollary}
\begin{proof}
For $T \in \,${\bf BO}, we have $T(\iota_1,\iota_2) =\widetilde{k}_{\#}(\theta )$, for $\theta \in \pi _{q_3}(\Sigma (\bar M_1 \wedge \bar M_2);G_3)$. Conversely given $\theta $, we define $T$ by $T(\iota_1,\iota_2) =\widetilde{k}_{\#}(\theta )$. It suffices to prove that $T$ is basic, that is, $j\widetilde{k}\theta \simeq 0$, by Corollary \ref{equivs}. But $j\widetilde{k}
\simeq j \lambda \mu \simeq 0$ since $ j \lambda \mu $ factors through $C\Sigma (\bar M_1 \wedge \bar M_2)$ (see the diagram in the proof of Theorem \ref{theta}).
\end{proof}

\noindent We have seen that if $T$ is basic and $q_3 <q_1+q_2 +\min(q_1,q_2)-3$, then $T$ is bi-additive. The following corollary gives additional properties with this hypothesis.

\begin{corollary} \label{properties} If \,$T \in \,${\bf BO} and $q_3 <q_1+q_2 +\min(q_1,q_2)-3$, then

\begin{enumerate} 
\item $T(\alpha , \beta ) =0$ if $X$ is an H-space;
\item $ET(\alpha , \beta ) =0$, where $E: \pi _{q_3}(X;G_3) \to \pi _{{q_3}+1}(\Sigma X;G_3)$ is the suspension homomorphism;
\item If $q_3 \le q_1 +q_2 -3$, then $T(\alpha , \beta ) = 0$, for all $\alpha $ and $\beta$.
\end{enumerate} 

\end{corollary}

\begin{proof} The generalized Whitehead product satisfies the first two properties (see \cite{a}, p.\,13), and so (1) and (2) follow from
Theorem \ref{theta}. 
Property (3) follows since the dimension of $M(G_3,q_3)$ is $\le q_3 +1$ and $\Sigma (\bar M_1\wedge \bar M_2) $ is $(q_1 + q_2 -2)$-connected, and so $\theta _T $ is nullhomotopic.
\end{proof}
\noindent The inequality in property (3) of Corollary \ref{properties} cannot be improved. That is, there are non-trivial operations
with $q_3 =q_1 + q_2 -2$. To see this, let $q= q_1 +q_2 $, let $G_1 = G_2 =\mathbb{Z}$ and $G_3 =\mathbb{Z}_k$, for some integer $k>1$ (so that $M_1 = S^{q_1}$ and $M_2 = S^{q_2}$). Set $(Y,X) = (S^{q_1}\times S^{q_2}, S^{q_1}\vee S^{q_2})$ and consider
the homomorphisms
$$\xymatrix{\mathrm{Ext}(\mathbb{Z}_k,\pi_q(Y,X) )\ar[r]^{\lambda } & \pi_{q-1}(Y,X;
\mathbb{Z}_k)\ar[r]^-{\partial}  
& \pi_{q-2}(X;\mathbb{Z}_k), }$$
where $\lambda $ is the monomorphism of the Universal Coefficient Theorem and $\partial $ is the boundary homomorphism. Furthermore, $\partial $ is a monomorphism. This follows from the exact homotopy sequence
with coefficients $\mathbb{Z}_k$ of $(S^{q_1}\times S^{q_2},S^{q_1}\vee S^{q_2})$ and the fact that 
$j_{\#} : \pi _i(S^{q_1}\vee S^{q_2}) \to \pi _i(S^{q_1}\times S^{q_2})$ has a right inverse.
In addition, $\pi_q(Y,X)$ is isomorphic to $\mathbb{Z}$ and so the Ext term
is $\mathrm{Ext}(\mathbb{Z}_k, \mathbb{Z}) =\mathbb{Z}_k$. Thus the monomorphism $\partial \lambda $ maps these $k$ elements
into $\pi_{q-2}(M_1\vee M_2;\mathbb{Z}_k)$ and hence determines $k$ basic homotopy operations. Since $q-2 <q_1+q_2 +\min(q_1,q_2)-3$,
all of these operations are bi-additive and have the properties listed in Corollary \ref{properties}. We shall refer to these operations of type $\{ \mathbb{Z}, \mathbb{Z}, \mathbb{Z}_k ; q_1, q_2, q_1 +q_2 -2  \}$ as {\bf Ext }operations. For more about them, see \S 6.

\section{Whitehead and Torsion products}

\noindent We next define a class of basic operations. The purpose is to unify Hilton's treatment of Whitehead products and Torsion products
in (\cite{h}, pp.\,110-120). Let $T$ be a basic operation
of type $\{ G_1, G_2, G_3 ; q_1, q_2, q_3  \}$
and let $\omega \in \pi _{q_3}(M_1\vee M_2;G_3)$ be the universal element for 
$T$. Furthermore, let $(Y,X) = (M_1\times M_2, M_1\vee M_2)$ and $Z=M_1\wedge M_2$.
We give three conditions:
\begin{enumerate} 

\item $H_{q_3 +1}(Z) \approx  G_3$, and we let $\phi : G_3 \to H_{q_3+1}(Z)=H_{q_3 +1}(M_1 \wedge M_2) $ be the isomorphism of the K\"unneth Theorem (see Remark 5.2 (1)).

\item $T$ is basic, and so there is a unique $\xi \in \pi _{q_3+1}(Y,X;G_3)$ such that $\partial (\xi )= \omega $,
\item There is a homomorphism $\widehat{\eta}$ which is defined by the following diagram
$$
\xymatrix{\pi _{q_3+1}(Y,X;G_3) \ar[r]^-{p_{\#} } 
 \ar@/_1pc/[rr]_{\widehat{\eta }} &   \pi _{q_3+1}(Z;G_3) \ar[r]^-{\eta }
  & \mathrm{Hom} (G_3, \pi _{q_3+1}(Z)), }
$$
where $p:Y =M_1\times M_2\to Z=M_1\wedge M_2$ is the  projection and $\eta $ is the epimorphism in the Universal Coefficient Theorem. 
Then the third condition is that the
following composition is equal to the identity map
$$\xymatrix{G_3 \ar[r]^-{\widehat{\eta }(\xi )} & \pi _{q_3+1}(M_1\wedge M_2)
\ar[r]^-h & H _{q_3+1}(M_1\wedge M_2) \ar[r]^-{\phi ^{-1}} & G_3,}$$ where $h$ is the Hurewicz homomorphism.

\end{enumerate}

\begin{definition}\label{special} Any operation which satisfies these three conditions will be called a {\bf special
operation}. 
\end{definition}

\noindent \begin{rmk} \label{specrmks} 
\begin{enumerate}
{\em \item  From the K\"unneth Theorem we have that there are two possibilities for a special operation:
\newline (a) $q_3 =q_1 + q_2 -1$ and $G_3= G_1\otimes G_2$. In this case assume
$q_1,q_2 \ge 3$. 
\newline (b) $q_3 =q_1 + q_2 $ and $G_3= G_1* G_2 = 
\mathrm{Tor}(G_1,G_2)$. In this case assume
$q_1,q_2 \ge 4$.
\item We comment on the third condition. The set of special operations of a given type may be empty since the third condition may not be satisfied (assuming that the first two are). Since $q_3 <q_1+q_2 +\min(q_1,q_2)-3$, $p_{\#} $ is an epimorphism and so $\widehat{\eta } $ is an epimorphism. If, in addition, $h$ is an isomorphism, then the set $\widehat{\eta }^{-1}(h^{-1}\phi)$ is non-empty and so the set of special operations equals $\partial \widehat{\eta }^{-1}(h^{-1}\phi ).$  In the general case ($h$ not
necessarily an isomorphism), $\partial (\xi )= \omega $ and $\widehat{\eta}(\xi ) = \eta (p_{\#}\xi ) = \eta (p\xi)$. There is a commutative diagram
$$
\xymatrix{\pi _{q_3 +1}(M(G_3,q_3 +1))\ar[d]^{h'}\ar[r]^{(p\xi)_{\# }} &\pi _{q_3 +1}(M_1\wedge M_2)\ar[d]^h
\\ G_3= H _{q_3 +1}(M(G_3,q_3 +1))\ar[r]^-{(p\xi )_*} & H_{q_3 +1}(M_1\wedge M_2),}
$$
and $\eta (p\xi ) =(p\xi)_{\#} h'^{-1}$. Therefore $$ \phi = h\widehat{\eta }(\xi ) = h\eta (p\xi ) =h (p\xi )_{\#}h'^{-1} =  (p\xi )_*  h' h'^{-1}=  (p\xi )_* .$$
Thus the third condition is $$ p_*\xi _* =(p\xi )_*= \phi  :G_3= H _{q_3 +1}(M(G_3,q_3 +1))\to H_{q_3 +1}(M_1\wedge M_2).$$}
\end{enumerate}
\end{rmk}

\noindent In \cite{h}, Hilton defined two classes of binary homotopy operations, the Whitehead products and the Torsion products. 
We give slightly different definitions which are equivalent to Hilton's definition. One difference is that we apply the Universal Coefficient Theorem
to $M_1\wedge M_2$ instead of to the pair $(M_1\times M_2, M_1 \vee M_2)$. In the relevant degrees the homotopy groups of the
two are isomorphic. A second difference concerns the existence of the isomorphism $\phi $. In 
(\cite{h}, pp.\,110, 115), $H_{q_3 +1}(M_1 \wedge M_2)$ is identified with $G_3$, whereas we make the isomorphism  explicit.
\medskip

\noindent Whitehead products (with coefficients) may now be defined as special operations $T$ satisfying 1(a) in Remark \ref{specrmks} with $G_3 = G_1 \otimes G_2$ and $q_3  =q_1+ q_2 -1$. They are called Whitehead products of type $\{G_1, G_2; q_1, q_2 \}$. It is clear that $\widehat{\eta} :\pi _{q_3 +1}(M_1\times M_2, M_1\vee M_2;G_3) \to \mathrm{Hom} (G_3, \pi _{q_3+1}(M_1\wedge M_2))$ is onto and $h: \pi _{q_3+1}(M_1\wedge  M_2)\to 
H _{q_3+1}(M_1\wedge M_2) $ is an isomorphism by the Hurewicz Theorem. Therefore if $\omega $ is an element of the non-empty
set $\partial \widehat{\eta}^{-1}(h^{-1}\phi )$, then by definition $\omega $ is the universal element of a
{\bf Whitehead product} $T$. There can be several different Whitehead products, in fact, since $\partial $ is one-one, the number of
Whitehead products is just the cardinality of the set $\Ker \widehat{\eta } = \Ker \eta =  \mathrm{Ext} (G_3, \pi _{q_3+2}(M_1\wedge M_2))$ by the Universal Coefficient Theorem. Moreover, if $T $ is a Whitehead product, then there exists a unique $\theta _T\in \pi _{q_3}(\Sigma (\bar M_1 \wedge \bar M_2);G)$
such that $T(\alpha ,\beta) = [\alpha , \beta] \, \theta _T $ by Theorem 
\ref{theta}. Thus each Whitehead product $ T$ satisfies bi-additivity and the properties listed in Corollary \ref{properties}.

\medskip

\noindent Torsion products are defined as special operations $T$ which satisfy 1(b) in Remark \ref{specrmks}, with $G_3=G_1*\,G_2= \mathrm{Tor} (G_1,G_2),$ $q_3=q_1 + q_2,$ and, in addition, $q_1,q_2 \ge 4$. They will be called Torsion products of type $\{G_1, G_2; q_1, q_2 \}.$
A {\bf Torsion product} 
is determined by an element $\tau \in \pi _{q_3}(M_1\vee M_2;G_3)$ such that $\tau = \partial (\zeta )$, for some $\zeta \in \pi _{q_3+1}(M_1\times M_2, M_1\vee M_2;G_3)$,
and such that $\phi ^{-1}h\widehat {\eta  } (\zeta )$ 
is the identity homomorphism of $G_3$ (see \cite{h}, p.\,115).
Note that the set of Torsion products may be empty, though not if $h$ is an 
isomorphism. In this case the number of Torsion products equals the cardinality of
$\Ker \widehat{\eta } = \Ker \eta =  \mathrm{Ext} (G_3, \pi _{q_3+2}(M_1\wedge M_2))$.
Furthermore, the hypotheses of
Theorem \ref{theta} are satisfied, and so there exists a unique $\theta_T \in \pi _{q_3}(\Sigma (\bar M_1 \wedge \bar M_2);G_3)$
such that $T(\alpha ,\beta) = [\alpha , \beta] \, \theta _T .$ As in the previous case, each Torsion product $T$ is bi-additive and satisfies the properties listed in Corollary \ref{properties}.
\medskip

\noindent  Many of these
properties for Whitehead and Torsion products have been proved in (\cite{h}, pp.\,111--113 and 116--119) directly from the definitions, though some of
our results (such as bi-additivity) are more general and the proofs are shorter (see \cite{h}, Theorems 12.3 and 12.6). 
\medskip

\noindent Since  $\pi _i(X; G \oplus G') \approx \pi _i(X; G )\oplus \pi _i(X;  G')$,
for any $i\ge 2$ and groups $G$ and $G',$ and since $G_1 *\,G_2 = 0$ if $G_1$ or $G_2$ is free-abelian, for the Torsion product we may restrict attention to the case when $G_1 =\mathbb{Z}_m$ and $G_2 =\mathbb{Z}_n$ are cyclic groups. Then $G_1*G_2 = \mathbb{Z}_d$, where $d$ is the greatest common divisor of $m$ and $n$. 
In (\cite{h}, pp.\,115--116) the following was proved: A Torsion product of type $\{\mathbb{Z}_m, \mathbb{Z}_n ; q_1, q_2  \}$
exists if and only if (1) $d$ is odd or (2) $m$ and $n$ are even and either $m $ or $n$ is a multiple of 4. 
In particular, a Torsion product exists if $m = n= p^k$, where $p$ is an odd prime and $k\ge 1$.
\medskip

\noindent
Another approach to Whitehead and Torsion products is suggested by Theorem \ref{theta}. In that theorem it is proved that many basic operations $T$  can be written as $T(\alpha , \beta ) = [\alpha , \beta]\, \theta _T$, for a unique
$\theta _T \in \pi _{q_3}(\Sigma (\bar M_1\wedge \bar M_2);G_3).$ 
This suggests that we define an operation by $T(\alpha , \beta ) = [\alpha , \beta]\, \theta 
$, for some $\theta \in \pi _{q_3}(\Sigma (\bar M_1\wedge \bar M_2);G_3)$.
\begin{prop}\label{specequiv} Let the operation $T$ of type $\{G_1, G_2, G_3 ; q_1, q_2, q_3 \}$ be defined by $T(\alpha , \beta ) = [\alpha , \beta]\, \theta $
and let $\phi : G_3 \to H_{q_3 +1}(M_1 \wedge M_2)
= H_{q_3 +1}(\Sigma \bar M_1 \wedge \Sigma \bar M_2)$ be the K\"unneth isomorphism.
Then $T$ is a special operation if and only if 
$(\Sigma \theta )_* =\sigma ^{-1}_* \phi , $ where $\sigma : \Sigma ^2(\bar M_1\wedge \bar M_2)\to \Sigma \bar M_1 \wedge \Sigma \bar M_2$ is defined in Theorem \ref{3.1}
\end{prop}
\begin{proof} First we show that $T$ is basic. Let $\omega $ be the universal element of $T$
and consider the homotopy-commutative diagram obtained from Theorem \ref{3.1} (see also 
the diagram in the proof of Theorem \ref{theta})
$$\xymatrix{M(G_3,q_3) \ar[r]^{\theta }\ar[d] & \Sigma(\bar M_1\wedge \bar M_2) \ar[r]^{\widetilde{k}} \ar[d] & \Sigma \bar M_1 \vee \Sigma \bar M_2 \ar[d]^j
\\ CM(G_3,q_3) \ar[r]^{C\theta } \ar[d] & C\Sigma(\bar M_1\wedge \bar M_2) \ar[r]^-{\Lambda ' } \ar[d]&  \Sigma \bar M_1 \times \Sigma \bar M_2 \ar[d]^p
\\ \Sigma M(G_3,q_3) \ar[r]^{\Sigma \theta } & \Sigma ^2(\bar M_1\wedge \bar M_2) \ar[r]^{\sigma } & \Sigma \bar M_1 \wedge \Sigma \bar M_2,
}$$
where $\Lambda ' = \Lambda (C\mu )$. Then $j\omega = j\widetilde{k}\theta \simeq 0$
and so $T$ is basic.
If $\xi =  \Lambda '(C\theta ) \in \pi _{q_3+1}(\Sigma \bar M_1 \times \Sigma \bar M_2, \Sigma \bar M_1 \vee \Sigma \bar M_2 ;G_3)$, then $\partial (\xi )= \omega $. Also $p(\xi )= \sigma (\Sigma \theta ),$ and so the third condition is 
$$\sigma _*(\Sigma \theta )_* =p_*\xi _* =  \phi , $$ which is equivalent to $(\Sigma \theta )_* =\sigma ^{-1}_* \phi . $
This completes the proof.
\end{proof}

\noindent Note that if $T,$ given by the universal element $\widetilde{k}\theta $, is a special operation, then $\theta  _*: H_{q_3}(M(G_3,q_3))  \to H_{q_3}(\Sigma (\bar M_1\wedge \bar M_2))$ is an isomorphism.
\medskip

\noindent By Remark \ref{specrmks}, there are only two possibilities for special operations. The following result is then a consequence
of Proposition \ref{specequiv}.
\begin{corollary} \label{two} Let $T$ be an operation of type $\{G_1, G_2, G_3 ; q_1, q_2, q_3 \}$ with universal element $\widetilde{k}\theta $
for some $\theta \in \pi _{q_3}(\Sigma (\bar M_1\wedge \bar M_2);G_3).$
\begin{enumerate} 
\item Let $q_3=q_1 +q_2-1 $ and $ G_3 = G_1 \otimes G_2$. Then $T$ is a Whitehead product if and only if $(\Sigma \theta )_* = \sigma _*^{-1} \phi .$
\item Let $q_3=q_1 +q_2 $ and $ G_3 = G_1 * G_2$. Then $T$ is a Torsion product if and only if $(\Sigma \theta )_* = \sigma _*^{-1} \phi .$
\end{enumerate}
\end{corollary}

\begin{rmk} {\em For Whitehead products ($q=q_1 +q_2 =q_3 +1 $ and $ G_3 = G_1 \otimes G_2$), we claim that this corollary can identify those $\theta \in \pi _{q-1}(\Sigma (\bar M_1\wedge \bar M_2);G_3)$ such that $\widetilde{k}\theta $ are
all the Whitehead universal elements. If $\omega \in \pi_{q-1}(M_1\vee M_2;G_3)$ is a universal element, then $\omega = \partial (\xi )$ for $\xi \in \pi_{q}(M_1\times M_2, M_1\vee M_2;G_3)$. Thus $p\xi  \in \pi_{q}(M_1\wedge M_2;G_3)$ and $\sigma ^{-1}(p\xi ) \in \pi _q (\Sigma ^2 (\bar M_1\wedge \bar M_2);G_3).$
Because the suspension homomorphism $E: \pi _{q-1} (\Sigma  (\bar M_1\wedge \bar M_2);G_3)\to \pi _q (\Sigma ^2 (\bar M_1\wedge \bar M_2);G_3)$ is an isomorphism, there exists a unique $\theta \in \pi _{q-1} (\Sigma  (\bar M_1\wedge \bar M_2);G_3)$ such that $\Sigma \theta = \sigma ^{-1}(p\xi )$.
We will show that $\widetilde{k} \theta = \omega .$ From the diagram in the proof of Proposition \ref{specequiv} we see that 
$\partial (\Lambda ' (C\theta )) = \widetilde{k}\theta $ and furthermore, with $p_{\#}: \pi _q(M_1\times M_2, M_1 \vee M_2;G_3)\to
\pi _q(M_1\wedge M_2;G_3)$, 
$$p_{\#}(\Lambda ' (C\theta ))= p(\Lambda ' (C\theta ))= \sigma (\Sigma \theta ) =\sigma \sigma^{-1}(p\xi ) =p_{\#}(\xi ).$$
Since $p_{\#}$ is an isomorphism, $\Lambda ' (C\theta ) =\xi ,$ and so
$$ \widetilde{k}\theta = \partial (\Lambda ' (C\theta ))= \partial (\xi )=\omega.$$ This establishes the claim.

\noindent A similar remark holds for the Torsion product.}
\end{rmk}

\noindent We next consider commutativity of special operations (see also \cite{h}, pp.\,113 -114 and 117-118).
\smallskip

\noindent Let $T$ be a special operation of type $\{G_1, G_2 ,G_3; q_1, q_2,q_3 \}$. Then $G_3 =G_1 \otimes G_2$ or $G_3 =G_1*G_2$ and we set 
$G'_3 =G_2 \otimes G_1$ or $G'_3 =G_2 * G_1$, accordingly. Furthermore, let $ t: G'_3 \to G_3$ be the switching isomorphism ($G_2\otimes G_1 \to G_1\otimes G_2$ or $G_2 * G_1 \to G_1* G_2$). Then there is a map $\tau : M(G'_3,q_3) \to M(G_3,q_3)$ such that $\tau _* =t.$  
\begin{prop} With $T$ a special operation as above, we define an operation $S$ by
 $$ S(\beta ,\alpha ) = (-1)^{\varepsilon } T(\alpha , \beta )\tau ,$$
for $\alpha \in \pi _{q_1}(W;G_1)$ and $\beta \in \pi _{q_2}(W;G_2)$ for any space $W,$ where $\varepsilon = q_1q_2$ when $G_3 =G_1 \otimes G_2$ and $\varepsilon = q_1q_2+1$ when $G_3 =G_1*G_2$. 
Then $S$ is a special operation of type $\{G_2, G_1 ,G_3'; q_2, q_1,q_3 \}$.
\end{prop}
\begin{proof} Let $(Y,X) =(M_1\times M_2,M_1 \vee M_2)$, $(Y',X')=(M_2\times M_1,M_2 \vee M_1)$, $Z =M_1 \wedge M_2$, and  $Z' =M_2 \wedge M_1$.
If $\rho :(Y,X) \to (Y',X')$ is the switching map, then
$\rho $ determines maps $\rho ': X \to X'$ and $\rho '' : Z \to Z' .$ There is a commutative diagram
$$\xymatrix{\pi _{q_3+1}(Y,X;G_3) \ar[r]^-{\widehat{\eta} }\ar[d]^-{\rho _{\#}\tau ^{*}} &  \mathrm{Hom} (G_3, \pi _{q_3+1} (Z)) \ar[r]^{h_*} \ar[d]^-{(\rho ''_{\#})_*t^* }&  \mathrm{Hom} (G_3, H _{q_3+1} (Z))  
\ar[d]^-{(\rho ''_{*})_*t^* }
\\ \pi _{q_3+1}(Y',X';G'_3) \ar[r]^-{\widehat{\eta '} } &  \mathrm{Hom} (G'_3, \pi _{q_3+1} (Z')) \ar[r]^{h'_*}  &  \mathrm{Hom} (G'_3, H _{q_3+1} (Z')).}$$
If $\omega \in \pi _{q_3}(X;G_3)$ is the universal element with $\omega = \partial (\xi )$ for
$\xi \in \pi _{q_3+1}(Y,X;G_3)$, then $h_* \widehat{\eta } (\xi )= \phi \in  \mathrm{Hom} (G_3, H_{q_3+1} (Z))$. 
We show that if $\xi ' = \rho _{\#} \tau ^{*}(\xi )$, then $h'_* \widehat{\eta '} (\xi ')= (-1)^{\varepsilon}\phi  '$,
where $\phi '  : G'_3\to H_{q_3+1} (Z')$ is the K\'unneth isomorphism.
We have
$$  h'_* \widehat{\eta '} (\xi ')=(\rho ''_{*})_*t^*(h_* \widehat{\eta } (\xi ))=(\rho ''_{*})_*t^*(\phi )=\rho ''_{*}\phi t.$$
It follows immediately from results in (\cite{h}, pp.\,114 and 118) that the following diagram is commutative
$$\xymatrix{G_3\ar[r]^-{(-1)^{\varepsilon}t^{-1}}\ar[d]_{\phi } & G_3' \ar[d]^{\phi '}
\\H_{q_3+1} (Z) \ar[r]^{\rho ''_* }& H _{q_3+1} (Z').}$$
Therefore
$$h'_* \widehat{\eta '} (\xi ')= \rho ''_{*}\phi t =(-1)^{\varepsilon}\phi'  t^{-1}t =(-1)^{\varepsilon}\phi' .$$
We set $\omega ' = \partial ' (\xi '),$ where $\partial ' :\pi _{q_3+1}(Y',X';G_3')
\to \pi _{q_3}( X';G_3').$ Thus $(-1)^{\varepsilon}\omega ' $ is the universal element of an
operation $S$ of type $\{ G_2,G_1,G'_3;q_2,q_1,q_3)$. Note that 
$$ \rho ' \omega \tau =\rho '_{\# }  \tau ^*(\partial \xi ) = \partial '(\rho _{\#}   \tau ^*( \xi ))=
\partial '( \xi ') =\omega '.$$ Let $j_1: M_2 \to M_2 \vee M_1$ and $j_2: M_1 \to M_2 \vee M_1$
be inclusions.Then 
$$S(j_1,j_2) = (-1)^{\varepsilon} \rho ' T(\iota _1, \iota _2)\tau = (-1)^{\varepsilon} T(j_2, j_1)\tau ,$$
and so $$ S(\beta ,\alpha ) = (-1)^{\varepsilon } T(\alpha , \beta )\tau ,$$
for $\alpha \in \pi _{q_1}(W;G_1)$ and $\beta \in \pi _{q_2}(W;G_2)$. The conclusion of the proposition now follows.
\end{proof}
Note that if the operation $T$ is unique, then there is the following anti-commutative rule
$$T(\beta ,\alpha ) = (-1)^{\varepsilon } T(\alpha , \beta )\tau .$$

\begin{corollary} 
\begin{enumerate}
\item If $T$ is a Whitehead product of type $\{G_1, G_2 ; q_1, q_2 \}$,
then the special operation $S$ defined by $ S(\beta ,\alpha ) = (-1)^{q_1q_2} T(\alpha , \beta )\tau $
is a Whitehead product of type $(q_2,q_1;G_2,G_1).$
\item If $T$ is a Torsion product of type $\{G_1, G_2 ; q_1, q_2 \}$,
then the special operation $S$ defined by $ S(\beta ,\alpha ) = (-1)^{q_1q_2 +1} T(\alpha , \beta )\tau $
is a Torsion product of type $(q_2,q_1;G_2,G_1).$
\end{enumerate}
\end{corollary}

\section{Concluding remarks and results}

\begin{enumerate}
\item{Neisendorfer's approach}

\noindent We consider operations with coefficients $\mathbb{Z}_{p^k}$, $p$ an odd prime and $k\ge 1.$
As mentioned earlier, any operation with finite groups of coefficients of odd order can be expressed in terms of operations with these coefficients. 
\noindent Let $G_1 = G_2 = G_3 = \mathbb{Z}_{p^k} ,$ so that $ G_1 \otimes G_2 = G_1* G_2= \mathbb{Z}_{p^k} ,$ and let $M(i)= 
M(\mathbb{Z}_{p^k},i)$. Then 
Neisendorfer proved that there is a homotopy equivalence $\delta : M(q-2)\vee M(q -1)\to M(q_1-1) \wedge M(q_2 -1) =\bar M_1\wedge \bar M_2$, where $q = q_1 +q_2$ (\cite{n}, p.\,167). 
We suspend and obtain (after the identification of $\Sigma (M(q-2)\vee M(q -1))$ with $M(q-1)\vee M(q)$) a homotopy equivalence
$\delta ': M(q-1)\vee M(q)\to \Sigma (\bar M_1\wedge \bar M_2) .$ If $j_2:M(q)\to  M(q-1)\vee M(q)$ is the inclusion, then we set $\theta = \delta 'j_2: M(q) \to 
\Sigma (\bar M_1\wedge \bar M_2)$ and define an operation ${\bf T}$ of type $\{\mathbb{Z}_{p^k}, \mathbb{Z}_{p^k}, \mathbb{Z}_{p^k} ; q_1, q_2, q_1+ q_2  \}$ by ${\bf T} (\alpha , \beta ) =[\alpha ,\beta ] \theta .$
This operation was originally defined in (\cite{n}, Section 6.3) where many of its properties were studied in detail. It was referred to as a Whitehead product. This may seem puzzling at first since the degrees of ${\bf T}$ are not those of a Whitehead product. But Neisendorfer used a definition of homotopy groups with coefficients which is different from the one we use. He defined them by means of co-Moore spaces (also called Peterson spaces), that is, simply-connected spaces $C(G,n)$ with a single, non-vanishing reduced cohomology group $G$ in degree $n$.  Then these homotopy groups with coefficients, which we shall denote by $\pi _n' ,$ are defined by $\pi _n'(X;G) = [C(G,n),X]$. Clearly $C( \mathbb{Z}_{p^k} ,n) = M( \mathbb{Z}_{p^k} ,n-1)$ and so $\pi _n(X; \mathbb{Z}_{p^k} ) =\pi _{ n+1}'(X; \mathbb{Z}_{p^k} )$. The product {\bf T} then becomes a function $\pi _{ q_1+1}'(X; \mathbb{Z}_{p^k} )\times \pi _{ q_2+1}'(X; \mathbb{Z}_{p^k} ) \to \pi _{ q_1 +q_2 +1}'(X; \mathbb{Z}_{p^k} )$ which are the correct degrees for a Whitehead product.
\medskip

\item Ext operations

\noindent We return to the Ext operations introduced at the end of \S 4 and provide a simple interpretation of them. For $\alpha \in \pi _{q_1}(X)$ and $\beta \in \pi _{q_2}(X)$, let $[\alpha , \beta] \in \pi _{q_1 +q_2-1}(X)$ be the ordinary Whitehead product (that is, the generalized Whitehead product with $\bar M_1=S^{q_1-1}$ and $\bar M_2=S^{q_2-1}$), let $M_{k,j} $ be the Moore space $M(\mathbb{Z}_k,j)$,
with $k\ge 2$ and $j\ge 3$, and let $q=q_1 +q_2$. Then $M_{k,q-2} $ is the mapping cone 
$S^{q-2}\cup _{\bf k} CS^{q-2}$, where ${\bf k}: S^{q-2}\to S^{q-2}$ is a map of degree k. A projection 
$p: M_{k,q-2} \to S^{q-1}$ is obtained by collapsing $S^{q-2} \subseteq M_{k,q-2}$ to a point and
$$p \in \pi _{q-2}(S^{q-1};\mathbb{Z}_k) \approx \mathrm{Ext}(\mathbb{Z}_k, \pi _{q-1}(S^{q-1}))\approx \mathbb{Z}_k.$$
By applying $[-,S^{q-1}]$ to the sequence $ \xymatrix{S^{q-2}\ar[r] &M_{k,q-2} \ar[r]^p& S^{q-1}}$
we obtain an exact sequence of homotopy groups,
from which it follows that $p$ is a generator of the group.
\begin{prop} The set $\{ [\iota _1, \iota _2](ip)\, | \, i=0,1,...,k-1 \}$ is equal to the set of universal elements of the k Ext operations. In particuliar, if $T$ is an Ext operation and $\alpha \in \pi _{q_1}(X)$ and 
$\beta \in \pi _{q_2}(X)$, Then $T(\alpha , \beta) = [\alpha ,\beta ](ip),$ for some $i \in \{0,1,...,k-1\}$.
\end{prop}

\noindent The proof is omitted, though we make a few comments about it. One shows that $ [\iota _1, \iota _2]p$ is basic
as in the proof of Proposition \ref{specequiv} by taking $\xi = \Lambda '(Cp) \in \pi _{q-1}(S^{q_1} \times S^{q_2}, S^{q_1} \vee S^{q_2}; \mathbb{Z}_k)$ so $\partial \xi = [\iota _1, \iota _2]p$.
Then $\eta (\xi) =0$ and so $\xi $ is in Ker $\eta $. Lastly, the set $\{ [\iota _1, \iota _2](ip)\, | \, i=0,1,...,k-1 \}$ has k elements since $\widetilde{k}_{\#}$ is one-one as in Theorem \ref{theta}.

\item Smash product of two Moore spaces
\begin{theorem}\label{wedge} Let $G_1$ and $G_2$ be finitely-generated abelian groups such that neither $G_1$  nor $G_2$
has 2-torsion. Then there is a homotopy equivalence
$$ M(G_1,q_1) \wedge  M(G_2,q_2) \equiv M(G_1 \otimes G_2, q_1 +q_2) \vee M(G_1 * G_2, q_1 +q_2 +1).$$
\end{theorem}
\begin{proof} Let $M_i = M(G_i,q_i)$ for $i=1,2$, let $q=q_1 +q_2$, and let $G_3 =G_1\otimes G_2$ and 
$\bar{
G_3 } = 
G_1*G_2$. It is easily seen (e.g., by a homology decomposition) that there is a map $l : M(\bar{
G_3 }, q)\to M(G_3, q)$ such that  
$ M(G_1,q_1) \wedge  M(G_2,q_2)$ has the homotopy type of the mapping cone $M(G_3, q)\cup _l CM(\bar{
G_3 }, q)$ and that $l$ is homologically trivial. We shall show that $l=0$. Now $l \in \pi _q (M(G_3, q); \bar{
G_3 }) $ and, with $M = M(G_3,q)$, we consider 
$$\xymatrix{\pi _q (M;\bar{
G_3 })\ar[r]^-{\eta } & \mathrm{Hom} (\bar{
G_3 }, \pi _q (M))\ar[r]^-{h_*}
&\mathrm{Hom}(\bar{
G_3 }, H _q (M)),}$$
where $\eta $ is the Universal Coefficient homomorphism and $h_*$ is induced by the Hurewicz isomorphism $h$. Then
$$ h_*(\eta (l)) = hl_{\#} = l_* =0 ,$$
and so $\eta (l) =0.$ Therefore by exactness of the Universal Coefficient Theorem, $l = \lambda (\tilde{l}),$ for
$\tilde{l}\in \mathrm{Ext}(\bar{G_3}, \pi _{q+1}(M(G_3,q)))$. It suffices to show that $\tilde{l} =0$. 
We set $E = \mathrm{Ext}(\bar{G_3}, \pi _{q+1}(M(G_3,q)))$ and show that 
$E=0$.
We write $G_i =F_i \oplus T_i$, $i=1, 2$,
where $F_i$ is a free-abelian group and $T_i$ is a finite torsion group. Then $E = \mathrm{Ext}(T_1 * T_2,A\oplus B \oplus C \oplus D)$, where $A=\pi _{q+1}(M(F_1\otimes F_2,q))$, $B = \pi _{q+1}(M(F_1\otimes T_2,q))$, $C=\pi _{q+1}(M(T_1\otimes F_2,q))$, and $D = \pi _{q+1}(M(T_1\otimes T_2,q))$. Then each of 
$F_1\otimes T_2$, $T_1\otimes F_2$, and $T_1\otimes T_2$ is a finite direct sum of cyclic
groups of order a power of an odd prime. But $\pi _{m+1}(M(\mathbb{Z}_n,m)) =0 $ if $n$ is odd (\cite{b}, p.\,268).
Therefore $B=C= D= 0$. Thus $ E =\mathrm{Ext}(T_1*T_2, A) =  \mathrm{Ext}(T_1*T_2, \pi _{q+1}(M(F_1\otimes F_2,q)))$.
Now $F_1\otimes F_2$ is a direct sum of finitely many copies of $\mathbb{Z}$ and so $M(F_1\otimes F_2,q)$ is a wedge of finitely many $q$-spheres $S^q_i$. Hence $A=\pi _{q+1}(M(F_1\otimes F_2,q))$ is a direct sum of terms
$\pi _{q+1}( S^q_i)$, that is, a direct sum of copies of $\mathbb{Z}_2$. Therefore $ E =\mathrm{Ext}(T_1*T_2, A) =0$
and so $l=0$. It follows that the mapping cone is a wedge of $M(G_1 \otimes G_2, q_1 +q_2) $ and $M(G_1 * G_2, q_1 +q_2 +1).$ This completes the proof.
\end{proof}
\begin{rmk} {\em Theorem \ref{wedge} holds if either $G_1$ or $G_2$ has 2-torsion (but not both). For definiteness suppose that $G_1$ has 2-torsion and $G_2$ does not. Then $T_1\otimes F_2$ is a finite direct sum of cyclic
groups of order a power of a prime including the prime 2. Thus $C=\pi _{q+1}(M(T_1\otimes F_2,q))$ is a finite direct sum of copies of $\mathbb{Z}_2$ (\cite{b}, p.\,268) and it follows that $\mathrm{Ext}(T_1*T_2, C) =0$.}
\end{rmk}

\item  Moore vs.\,co-Moore spaces

\noindent As a final comment we observe that there are advantages and disadvantages to using either co-Moore spaces or Moore spaces for coefficients. Co-Moore spaces are the dual of Eilenberg-MacLane spaces within the context of Eckmann-Hilton duality, where homotopy groups and cohomology groups are considered dual to each other, but co-Moore spaces do not exist for every group $G$ \cite{kw}. Moore spaces exist for every group, but they are not dual to Eilenberg-MacLane spaces.

\end{enumerate} 


\end{document}